\documentclass[12pt,a4paper,reqno]{amsart}

\usepackage{amssymb}
\usepackage{amsthm}
\usepackage{amsmath}

\usepackage[hdivide={2.8cm,,2.8cm}, vdivide={4cm,,3.5cm}]{geometry}

\newtheorem{thm}{Theorem}
\newtheorem{cor}[thm]{Corollary}

\newtheorem{knownthm}{Theorem}

\theoremstyle{definition}

\newtheorem{knowndefinition}[knownthm]{Definition}
\newtheorem{rem}{Remark}[section]

\newcommand{\closure}{\overline}

\newcommand{\round}{\partial}
\newcommand{\CC}{\widehat{\mathbb{C}}}
\newcommand{\C}{\mathbb{C}}
\newcommand{\D}{\mathbb{D}}

\newcommand{\A}{\mathcal{A}}
\renewcommand{\SS}{\mathcal{S}}

\newcommand{\epf}{\phantom{} \hfill$ \Box$}
\newcommand{\no}{\noindent}
\newcommand{\dstyle}{\displaystyle}
\renewcommand{\Re}{\textup{Re}\,}

\title[Ahlfors's quasiconformal extension condition]{Ahlfors's quasiconformal extension condition and $\Phi$-likeness}
\author[I. Hotta]{Ikkei Hotta}
\address{Department of Mathematics, University of W\"urzburg, Emil-Fischer-Stra\ss e 40 97074 W\"urzburg, Germany}
\email{ikkeihotta@gmail.com}
\date{\today}

\begin{document}

\maketitle

\vspace{-15pt}

\begin{abstract}
The notion of $\Phi$-like functions is known to be a necessary and sufficient condition for univalence.
By applying the idea, we derive several necessary conditions and sufficient conditions for that an analytic function defined on the unit disk is not only univalent but also has a quasiconformal extension to the Riemann sphere, as generalizations of well-known univalence and quasiconformal extension criteria, in particular, Ahlfors's quasiconformal extension condition.
\end{abstract}

\

\

\section{Introduction}

Let $\C$ be the complex plane, $\D_{r} := \{z \in \C : |z| < r\}$ for $r >0$ and $\D := \D_{1}$.
We denote by $\A$ the family of functions $f (z) = z + \sum_{n = 2}^{\infty} a_{n}z^{n}$ analytic on $\D$ and $\SS$ the subclass of $\A$ whose members are univalent, that is, one-to-one on $\D$.
For standard terminology in the theory of univalent functions, see for instance \cite{PoM:1975} and \cite{Goodman:1983ab}.
Let $k \in [0,1)$ be a constant. 
Then a homeomorphism $f$ of $G \subset \C$ is said to be \textit{$k$-quasiconformal} if $\round_{z} f$ and $\round_{\bar{z}} f$ in the distributional sense are locally integrable on $G$ and fulfill $|\round_{\bar{z}} f| \leq k |\round_{z} f|$ almost everywhere in $G$.
If we do not need to specify $k$, we will simply call that $f$ is \textit{quasiconformal}.

We begin our argument by observing a fundamental composition property of analytic functions.
Let $f,\,g \in \A$ and $Q$ be an analytic function defined on $f(\D)$ which satisfies $g = Q \circ f$.
Then a necessary and sufficient condition for univalence of $g$ on $\D$ is that $f$ and $Q$ are univalent on each domain.
Let us apply this fact to derive a univalence criterion.
In order to demonstrate it, we set one example with a condition for \textit{$\lambda$-spirallike functions}, i.e., $\Re \{ e^{- i \lambda} zg'(z) / g(z) \} > 0$ is satisfied for all $z \in \D$, where $\lambda \in (-\pi/2, \pi/2)$.
Note that all $\lambda$-spirallike functions are univalent on $\D$.
In view of the relationship $g = Q \circ f$, the condition of $\lambda$-spirallikeness of $g$ is equivalent to
	\begin{equation}\label{phi-like}
	\Re \frac{zf'(z)}{\Phi(f(z))} > 0,
	\end{equation}
where $\Phi(w) = e^{i\lambda} Q(w)/Q'(w)$.
We conclude that \eqref{phi-like} is a sufficient condition of univalence of $f$.
This is an essential idea of the notion of ``$\Phi$-like functions''.

\begin{knowndefinition}
A function $f \in \A$ is said to be \textit{$\Phi$-like} if there exists an analytic function $\Phi$ defined on $f(\D)$ such that \eqref{phi-like}
holds for all $z \in \D$.
\end{knowndefinition}

\begin{rem}
The inequality \eqref{phi-like} implies $\Phi(0) = 0$ and $\Re \Phi'(0) > 0$.
\end{rem}

The notion of $\Phi$-likeness also turns out a necessary condition for univalence of $f$.
In fact, if $f$ is univalent in $\D$ then we can define $\Phi$ by means of $Q := g \circ f^{-1}$, where $g$ is a spirallike function.
Consequently, we obtain the following:

\begin{knownthm}
A function $f \in \A$ is univalent in $\D$ if and only if $f$ is $\Phi$-like.
\end{knownthm}

\begin{rem}
If we choose $\Phi(w) = e^{i \lambda} w$ then it immediately follows the condition for $\lambda$-spirallikeness.
\end{rem}

The concept of $\Phi$-like function was introduced by Kas'yanyuk \cite{Kasyanyuk:1959} and Brickman \cite{Brickman:1973} independently. The reader is referred to \cite[$\S$7]{Avkhadiev:1975} which contains some more information about $\Phi$-like functions.
The above instructive characterization of $\Phi$-like functions is due to Ruscheweyh \cite{Ruscheweyh:pre01}. 
Furthermore, he gave the following two generalizations of well-known univalence conditions by the same technique as it:

\begin{knownthm}[Generalized Becker condition \cite{Ruscheweyh:pre01}]\label{gbecker}
Let $f \in \A$.
Then $f$ is univalent if and only if there exists an analytic function $\Omega$ on $f(\D)$ such that
	\begin{equation}\label{gbeckereq}
	(1-|z|^{2})
	\left|
	\frac{zf''(z)}{f'(z)} + zf'(z) \Omega(f(z))
	\right|
	\leq 1
	\end{equation}
for all $z \in \D$.
\end{knownthm}

\begin{knownthm}[Generalized Bazilevi\v c functions \cite{Ruscheweyh:pre01}]\label{gbazil}
Let $f \in \A$, $p(z)$ with $p(0) = p'(0)-1 = 0$ be starlike univalent in $\D$ and $s = \alpha + i \beta \in \C,\,\Re s > 0$.
Then $f$ is univalent in $\D$ if and only if there exists an analytic function $\Psi(w)$ on $f(\D)$ with $\Psi (0) \neq 0$ such that
	\begin{equation}\label{gbazileq}
	\Re
	\frac
	{f'(z) \left( f(z)/z\right)^{s-1}}
	{\left(p(z)/z\right)^{\alpha}}
	\Psi(f(z)) > 0
	\end{equation}
	for all $z \in \D$.
\end{knownthm}

\begin{rem}
The choices $\Omega \equiv 0$ and $\Psi \equiv e^{i\alpha}$ correspond to the original univalence conditions due to Becker \cite{Becker:1972} and Bazilevi\v c \cite{Bazilevic:1955} respectively.
\end{rem}

Since Ref. \cite{Ruscheweyh:pre01} is not published, we outline proofs of Theorem \ref{gbecker} and Theorem \ref{gbazil} for convenience.
We can show the case \eqref{gbeckereq} from the fact that $g(z) = Q(f(z))$ with $Q''/Q' = \Omega$ satisfies original Becker's univalence condition, and the case \eqref{gbazileq} from that the function $g(z) = Q(f(z))$ with
$
	Q (w)= 
	\left(
	s
	\int_{0}^{w} t^{s-1} \Psi(t) dt
	\right)^{1/s}
	=
	\Psi(0) w + \cdots
$
is Bazilevi\v c and hence univalent in $\D$.
The other directions of Theorem \ref{gbecker} and Theorem \ref{gbazil} can be easily proved to define $\Omega$ and $\Psi$ by $Q(w) = g(f^{-1}(w)),\, w \in f(\D)$, where $g$ is a suitable function which satisfies Becker's condition or the Bazilevi\v c function, respectively.

The main aim of this paper is to derive several necessary conditions and sufficient conditions for that a function $f \in \A$ is univalent in $\D$ and extendible to a quasiconformal mapping to the Riemann sphere $\CC := \C \cup \{\infty\}$
as an application of Ruscheweyh's characterization of $\Phi$-like functions.
These results are based on well-known univalence and quasiconformal extension criteria.
For instance, the next theorem which is a generalization of Ahlfors's quasiconformal extension condition \cite{Ahlfors:1974} (see also \cite{Becker:1976}) will be obtained. 
Here $\SS(k),\,0 \leq k < 1,$ is the family of functions which are in $\SS$ and can be extended to $k$-quasiconformal mappings to $\CC$.

\begin{thm}\label{gbeckerqc}
Let $f \in \A$ and $k \in [0,1)$.
If there exists a $k' \in [0,k)$ and an analytic function $Q$ defined on $f(\D)$ which is univalent in $f(\D)$, has a $(k-k')/(1-kk')$-quasiconformal extension to $\CC$ and satisfies $Q'(0) \neq 0$ such that for a constant $c \in \C$ and for all $ z \in \D$ 
	\begin{equation}\label{gbeckerqceq}
	\left| c |z|^{2} + 
	(1-|z|^{2})
	\left\{
	\frac{zf''(z)}{f'(z)} + zf'(z) \Omega(f(z))
	\right\}
	\right|
	\leq k',
	\end{equation}
then $f \in \SS(k)$, where $\Omega = Q''/Q'$.
Conversely, if $f \in \SS(k)$ then there exists a $k' \in [0,1)$ and an analytic function $Q$ defined on $f(\D)$ which is univalent in $f(\D)$, has a $(k+k')/(1+kk')$-quasiconformal extension to $\CC$ and satisfies $Q'(0) \neq 0$ such that the inequality \eqref{gbeckerqceq} holds for a constant $c \in \C$ and for all $z \in \D$.
\end{thm}

\begin{rem} 
In Theorem \ref{gbeckerqc}, if the extended quasiconformal mapping of $Q$ does not take the value $\infty$ in $\C$ then $\SS(k)$ can be replaced by $\SS_{0}(k)$, where $\SS_{0}(k)$ is the family of functions which belong to $\SS(k)$ and can be extended to $k$-quasiconformal automorphisms of $\C$.
\end{rem}

In contrast to the case of univalent functions, it is not always true that if $g \circ f$ has a quasiconformal extension then so do $f$ and $g$ as well. 
This is the reason why the function $Q$ is required some bothersome assumptions in Theorem \ref{gbeckerqc}.
On the other hand, if we give a specific form of $Q$, then it can be obtained several new quasiconformal extension criteria which are of practical use.
We will discuss this problem in the last section.

\section{Preliminaries}

Let $f_{t}(z) = f(z,t) = \sum_{n=1}^{\infty}a_{n}(t)z^{n},\,a_{1}(t) \neq 0,$ be a function defined on $\D \times [0,\infty)$ and analytic in $\D$ for each $t \in [0,\infty)$, where $a_{1}(t)$ is a locally absolutely continuous function on $[0,\infty)$ and $\lim_{t \to \infty}|a_{1}(t)| = \infty$.
$f_{t}$ is said to be a \textit{L\"owner chain} if $f_{t}$ is univalent on $\D$ for each $t \in [0,\infty)$ and satisfies 
$f_{s}(\D) \subsetneq f_{t}(\D)$ for $0 \leq s < t < \infty$.

The following necessary and sufficient condition for L\"owner chains due to Pommerenke is well-known.

\def\labelenumi{\arabic{enumi}.}
\begin{knownthm}[\cite{Pom:1965}]\label{pom}
Let $0 < r_{0} \leq 1$.
Let $f(z,t)$ be a function defined above.
Then the function $f(z,t)$ is a L\"owner chain if and only if the following conditions are satisfied:
	\begin{enumerate}
	\item The function $f(z,t)$ is analytic in $\D_{r_{0}}$ for each $t \in [0,\infty)$, 
	locally absolutely continuous in $[0,\infty)$ for each $z \in \D_{r_{0}}$ and
		\begin{equation*}\label{pom2ineq}
		|f(z,t)| \leq k' |a_{1}(t)| \hspace{20pt} (z \in \D_{r_{1}},\,\textup{a.e.}~t \in [0,\infty))
		\end{equation*}
	for some positive constants $k'$.
	\item There exists a function $p(z,t)$ analytic in $\D$ for each $t \in [0,\infty)$ 
	and measurable in $[0,\infty)$ for each $z \in \D$ satisfying
		$$
		\Re p(z,t) > 0 \hspace{20pt} (z \in \D,\,t \in [0,\infty))
		$$
	such that
		\begin{equation*}\label{LDE}
		\round_{t}f(z,t) =z \round_{z}f(z,t) p(z,t)  \hspace{20pt} (z \in \D_{r_{1}},\,\textup{a.e.}~t \in [0,\infty)).
		\end{equation*}
	\end{enumerate}
\end{knownthm}

\begin{rem}\label{remforLC}
It is known that $a_{1}(t)$ is admitted to be a complex-valued function (\cite{Hotta:2010a}).
In addition, it should be noted here about constant terms of L\"owner chains.
If $f(z,t)$ is a L\"owner chain then $f(z,t) + c$ satisfies all the conditions of the definition of L\"owner chains and the sufficient conditions of Theorem \ref{pom} with a modification of $k'$, where $c$ is a complex constant which does not depend on $t$. 
For this reason here and hereafter we shall also treat such functions as L\"owner chains.
\end{rem}

The next theorem which is due to Becker plays a central role in our argument:

\begin{knownthm}[\cite{Becker:1972, Becker:1976}]\label{becker}
Suppose that $f_{t}(z) = f(z,t)$ is a L\"owner chain for which $p(z,t)$ in \eqref{LDE} satisfies the condition 
\begin{equation*}
\begin{array}{lrll}
p(z,t) \in U(k) &\!:=\!&
\dstyle
\left\{
w \in \C : \left|\frac{w-1}{w+1}\right| \leq k
\right\}\\[15pt]
&\!=\!&
\dstyle
\left\{
w \in \C: \left|w-\frac{1+k^{2}}{1-k^{2}}\right| \leq \frac{2k}{1-k^{2}}
\right\}
\end{array}
\end{equation*}
for all $z \in \D$ and almost all $t \in [0,\infty)$.
Then $f(z,t)$ admits a continuous extension to $\closure{\D}$ for each $t \geq 0$ and the map $\hat{f}$ defined by
\begin{equation*}\label{beckerineq}
\hat{f}(re^{i\theta}) =
\left\{
\begin{array}{ll}
\dstyle f(re^{i\theta},0), &\textit{if}\hspace{10pt}   r < 1, \\[5pt]
\dstyle f(e^{i\theta},\log r), &\textit{if}\hspace{10pt}   r \geq 1,
\end{array} 
\right.
\end{equation*}
is a $k$-quasiconformal extension of $f_{0}$ to $\C$.
\end{knownthm}

In other words, if $f_{t}$ is normalized by $f_{t}(0)=0$ then the above theorem gives a sufficient condition for $f_{0} \in \SS_{0}(k)$.

\section{Proof of Theorem \ref{gbeckerqc}}

Firstly we show the first part of Theorem \ref{gbeckerqc}.
Set
\begin{equation}\label{gbeckerLC}
	F(z,t) := 
		Q(f(e^{-t}z)) + 
		(1+c)^{-1}
		(e^{t} - e^{-t}) 
		z Q'(f(e^{-t}z))f'(e^{-t}z).
\end{equation}
We take into account that $1 + c \neq 0$ since the inequality \eqref{gbeckerqceq} implies $|c| \leq k' < 1$ (see \cite[Remark 1.1 and 1.2]{Hotta:2010b}).
Then we have
\begin{equation}\label{proof1}
\begin{tabular}{llll}
	$\dstyle
	\left|
	\frac{\round_{t}F(z,t) - z \round_{z}F(z,t)}{\round_{t}F(z,t) + z \round_{z}F(z,t)}
	\right|$\\[12pt]
	\hspace{20pt}$\dstyle
	=
		\left|
		e^{-2t} c + (1 - e^{-2t}) 
		\left\{e^{-t}z \frac{f''(e^{-t}z)}{f'(e^{-t}z)} + e^{-t}z f'(e^{-t}z)\Omega(f(e^{-t}z))
		\right\}
 		\right|
	$
\end{tabular}
\end{equation}

\no
where $\Omega = Q'' / Q'$.
The right-hand side of \eqref{proof1} is always less than or equal to $k'$ from \eqref{gbeckerqceq} and hence $g := Q \circ f$ can be extended to a $k'$-quasiconformal mapping to $\CC$ by Theorem \ref{pom} and Theorem \ref{becker}.
Since $Q$ has a $(k-k')/(1-kk')$-quasiconformal extension to $\CC$ we conclude that $f = Q^{-1} \circ g \in \SS(k)$.

The second part of Theorem \ref{gbeckerqc} easily follows to define $Q := g \circ f^{-1}$, where $g$ is an analytic function defined on $\D$ which satisfies original Ahlfors's $k'$-quasiconformal extension condition.
\epf

\section{Further results}

We can derive similar necessary conditions and sufficient conditions for quasiconformal extensions as Theorem \ref{gbeckerqc}.
We select one example out of a large variety of possibilities.
This is based on the Noshiro-Warschawski theorem \cite{Noshiro:1934, Warschawski:1935}.

\begin{thm}\label{NW}
Let $f \in \A$ and $k \in [0,1)$. 
If there exists a $k' \in [0,k)$ and an analytic function $Q$ defined on $f(\D)$ which is univalent in $f(\D)$, has a $(k-k')/(1-kk')$-quasiconformal extension to $\CC$ and satisfies $Q'(0) \neq 0$ such that for all $z \in \D$ 
	\begin{equation}\label{NWineq}
	f'(z)Q'(f(z)) \in U(k')
	\end{equation}
then $f \in \SS(k)$, where $U(k')$ is the disk defined in Theorem \ref{becker}.
Conversely, if $f \in \SS(k)$ then there exists a $k' \in [0,1)$ and an analytic function $Q$ defined on $f(\D)$ which is univalent in $f(\D)$, has a $(k+k')/(1+kk')$-quasiconformal extension to $\CC$ and satisfies $Q'(0) \neq 0$ such that the inequality \eqref{NWineq} holds for all $z \in \D$.
\end{thm}

\begin{proof}
Let us put
\begin{equation}\label{NWLC}
	F(z,t) := 
	Q(f(z)) + (e^{t} -1)z.
\end{equation}
Then calculations show that 
$$
	\frac{z \round_{z}F(z,t)}{\round_{t}F(z,t)}
	=
	\frac{1}{e^{t}} \,Q'(f(z)) f'(z) + \left(1-\frac1{e^{t}}\right).
$$
Following the lines of the proof of Theorem \ref{gbeckerqc} one can deduce all the assertions of Theorem \ref{NW}.
\end{proof}

Various similar results as Theorem \ref{NW} can be proved to choose the other univalence criterion and set a suitable L\"owner chain.
For example, the condition
$$
\frac{zf'(z)}{\Phi(f(z))} \in U(k')
$$
which is based on the definition of $\Phi$-like functions is given by the L\"owner chain
$$
	F(z,t) = 
	e^{t} Q(f(z)),
$$
or
$$
\frac
{f'(z) \left( f(z)/z\right)^{s-1}}
{\left(p(z)/z\right)^{\alpha}}
\Psi(f(z))
\in U(k')
$$
which is based on the definition of the Bazilevi\v c functions is given by
$$
	F(z,t) = 
	\left\{
	Q(f(z))^{s} + s (e^{t} - 1) p(z)^{\alpha} z^{i \beta}
	\right\}^{1/s},
$$
where $\Phi$ and $\Psi$ are functions defined in Section 1.

\section{Applications}

In this section we consider several applications of theorems we have obtained in previous sections, in particular, Theorem \ref{gbeckerqc} and Theorem \ref{NW}.
Two specific forms of the function $Q$ which is univalent on a certain domain and can be extended to a quasiconformal mapping to $\CC$ are given.

We remark that in the cases below $Q$ does not need to be normalized by $Q(0) =0$. 
For L\"owner chains $F(z,t)$ defined in \eqref{gbeckerLC} and \eqref{NWLC} we have $F(0,t) = Q(0)$ which implies that in both cases a constant term of $F(z,t)$ does not depend on $t$.
Hence, as we noted in Remark \ref{remforLC}, $F(z,t)$ is a L\"owner chain even though $Q(0) \neq 0$.
This fact allows us to avoid some technical complications.

\subsection{M\"obius transformations}

Let $Q_{1}$ be the M\"obius transformation given by
\begin{equation}\label{moebius}
Q_{1}(w) := \frac{\alpha w + \beta}{\gamma w+ \delta} 
\hspace{15pt} 
(\alpha,\,\beta,\,\gamma,\,\delta \in \C,\,\gamma \neq 0,\,\alpha\delta - \beta\gamma = 1).
\end{equation}
For a given $f \in \A$ we suppose that $-\delta/\gamma \notin f(\D)$, for otherwise $Q_{1}$ is no longer analytic on $f(\D)$.
Thus $Q_{1}$ is considered as a function which is analytic and univalent on $f(\D)$ and has a 0-quasiconformal extension to $\CC$.
We note that $Q_{1}$ is the unique function which it can be chosen as $Q$ in Theorem \ref{gbeckerqc} and Theorem \ref{NW} without any restrictions on the shape of $f(\D)$.

Simple calculations show that $Q_{1}'(w) = 1/(\gamma w + \delta)^{2}$ and $Q_{1}''(w)/Q_{1}'(w) = -2 /(w + (\delta/\gamma))$, and hence by defining $Q := Q_{1}$ we obtain the following new quasiconformal extension criteria as corollaries of Theorem \ref{gbeckerqc} and Theorem \ref{NW}:

\begin{cor}\label{cor1}
Let $f \in \A$ and $k \in [0,1)$. 
If $f$ satisfies
$$
	\left| c_{1} |z|^{2} + 
	(1-|z|^{2})
	\left\{
	\frac{zf''(z)}{f'(z)} - \frac{2 zf'(z)}{f(z) - c_{2}}
	\right\}
	\right|
	\leq k
$$
for some constants $c_{1},\,c_{2}\in \C,\,c_{2} \notin f(\D)$, and for all $z \in \D$, then $f \in \SS(k)$.
\end{cor}

\begin{cor}\label{cor2}
Let $f \in \A$ and $k \in [0,1)$. 
If $f$ satisfies
$$
\frac{f'(z)}{(\gamma f(z) + \delta)^{2}} \in U(k)
$$
for some constants $\gamma,\, \delta \in \C,\,\gamma \neq 0,\,-\delta/\gamma \notin f(\D)$, and for all $z \in \D$, then $f \in \SS(k)$.
\end{cor}

\begin{rem}
We assumed that $\gamma \neq 0$ in \eqref{moebius} because in the case when $\gamma =0$ the function $Q_{1}$ is an affine transformation and thus Corollary \ref{cor1} and Corollary \ref{cor2} are nothing but well-known quasiconformal extension criteria in \cite{Ahlfors:1974} and \cite{Sugawa:1999a}.
\end{rem}

\begin{rem}
In the above corollaries $\SS(k)$ cannot be replaced by $\SS_{0}(k)$, because $\gamma \neq 0$ which implies $Q_{1}$ does not fix $\infty$.
\end{rem}

\subsection{Sector domain}

We may set the function $Q$ under the assumption that the image of $\D$ under $f \in \A$ is contained in a quasidisk which has a special shape. 
For instance, we suppose that $f(\D)$ lies in the sector domain 
$$
\Delta(w_{0}, \lambda_{0}, a) := \{w \in \C : \pi\lambda_{0} < \arg(w-w_{0}) < \pi\lambda_{1},\,|\lambda_{1} - \lambda_{0}| <  a\}
$$ 
for $w_{0} \in \C\backslash f(\D)$, $\lambda_{0} \in [0,2)$ and $a \in (0,2)$. 
Then, we define $Q$ in Theorem \ref{gbeckerqc} and Theorem \ref{NW} by $Q_{2}$,
$$
Q_{2}(w) := \left(e^{-i\pi \lambda_{0}}(w-w_{0})\right)^{1/a},
$$
which maps $\Delta(w_{0}, \lambda_{0}, a)$ conformally onto the upper half-plane.
It is verified that $Q_{2}$ can be extended to a $|1-a|$-quasiconformal automorphism of $\C$ as follows:
Let us set 
$$
P_{1}(z) := z^{1/(2-a)},\hspace{10pt}P_{2}(z) := |z|^{(2-a)/a}\frac{z}{|z|},
$$
respectively.
Then the function $P$ defined by
$$
P(z) :=
\left\{
\begin{array}{ll}
\dstyle z^{1/a}, &\textit{if}\hspace{10pt}   z \in \Delta(0,0,a), \\[5pt]
\dstyle -( P_{2} \circ P_{1} )(e^{-\pi a} z), &\textit{if}\hspace{10pt}   z \in \closure{\Delta(0, a, 2-a)},
\end{array} 
\right.
$$
is a $|1-a|$-quasiconformal automorphism of $\C$. 
After composing proper Affine transformations we obtain the desired extension of $Q_{2}$.

Since $Q_{2}''(w)/Q_{2}'(w) = ((1/a)-1) /(w-w_{0})$, we deduce the following:

\begin{cor}\label{Q2tobe}
Let $f \in \A$ and $k \in [0,1)$. 
We assume that $f(\D)$ is contained in the sector domain $\Delta(w_{0}, \lambda_{0}, a)$.
If $f$ satisfies
$$
	\left| c |z|^{2} + 
	(1-|z|^{2})
	\left\{
	\frac{zf''(z)}{f'(z)} + \left(\frac1a -1\right)\frac{zf'(z)}{f(z) - w_{0}}
	\right\}
	\right|
	\leq k
$$
for all $z \in \D$, then $f \in \SS_{0}(\ell)$, where $\ell = (k + |1-a|)/(1+k|1-a|)$.
\end{cor}

To state the next corollary we shall put $Q_{3}(w) :=  Q_{2}(w)/ Q_{2}'(0)$, so that $Q_{3} '(w) = (1 - (w/w_{0}))^{(1/a)-1}$ and hence $f'(0)Q_{3} '(0) = 1 \in U(k)$ for any $k \in [0,1)$.

\begin{cor}\label{Q2toNW}
Let $f \in \A$ and $k \in [0,1)$. 
We assume that $f(\D)$ is contained in the sector domain $\Delta(w_{0}, \lambda_{0}, a)$.
If $f$ satisfies
$$
f'(z) \left(1 - \frac{f(z)}{w_{0}}\right)^{(1/a) -1} \in U(k)
$$
for all $z \in \D$, then $f \in \SS_{0}(\ell)$, where $\ell = (k + |1-a|)/(1+k|1-a|)$.
\end{cor}

As a special case of Corollary \ref{Q2toNW}, if we can choose $w_{0}=1$ and $a = 1/2$, then we will have a $(2k+1)/(k+2)$-quasiconformal extension criterion $f'(z) (1 - f(z)) \in U(k)$.

\subsection{Bounded functions}


Corollary \ref{Q2tobe} and Corollary \ref{Q2toNW} may seem to be less useful in practical situations because of their assumption that $f(\D)$ lies in a sector domain. 
It will, however, be used effectively when we investigate whether $f \in \SS$ is contained in $\SS_{0}(k)$ or not. 
Actually, it is enough to deal with only bounded components of $\SS$ in this problem because if $f$ is unbounded then $f \notin \SS_{0}(k)$.

We can easily find a precise sector domain which includes $f(\D)$.
If $f \in \A$ is bounded on $\D$, then there exists a constant $M := \sup_{z \in \D} |f(z)|$. 
Since $f(\D) \subset \D_{M}$, $f(\D)$ is contained in $\Delta(w_{0}, \lambda_{2}, 2 \arcsin (M/|w_{0}|))$ for $w_{0} \in \C\backslash\D_{M}$ and a suitable $\lambda_{2}$ which will be given below.
Of course one may choose the other disk $\{z \in \C : |z-z_{0}|<R\}$ which contains $f(\D)$ and put a sector by $\Delta(w_{0}, \lambda_{3}, 2\arcsin(R/|w_{0}-z_{0}|))$ for $w_{0} \in \C\backslash\{ |z-z_{0}|<R\}$ and $\lambda_{3} \in [0,2)$.
Here $\lambda_{3} =\arg(z_{0} -w_{0}) + \arg(\sqrt{|w_{0} - z_{0}|^{2} - R^{2}} - i R)$. 
$\lambda_{2}$ is given by the case when $z_{0}=0$ and $R=M$.

\section*{Acknowledgement}
The author expresses his deep gratitude to Professor Stephan Ruscheweyh for many useful and instructive discussions. 
He would like to thank Professor Toshiyuki Sugawa for valuable comments and suggestions.
Part of this work was done while the author was a visitor at the University of W\"urzburg under the JSPS Institutional Program for Young Researcher Overseas Visits.


\bibliographystyle{amsplain}
\def\cprime{$'$}
\providecommand{\bysame}{\leavevmode\hbox to3em{\hrulefill}\thinspace}
\providecommand{\MR}{\relax\ifhmode\unskip\space\fi MR }
\providecommand{\MRhref}[2]{%
  \href{http://www.ams.org/mathscinet-getitem?mr=#1}{#2}
}
\providecommand{\href}[2]{#2}


\end{document}